\documentclass[psamsfonts]{amsart}

\usepackage{amssymb,amsfonts}
\usepackage[all,arc]{xy}
\usepackage{enumerate}
\usepackage{mathrsfs}
\usepackage{textcomp}
\usepackage{phaistos}

\usepackage{cite}

\newtheorem{thm}{Theorem}[section]

\newtheorem{lem}[thm]{Lemma}

\theoremstyle{definition}

\newtheorem{fact}[thm]{Fact}

\theoremstyle{remark}
\newtheorem{rem}[thm]{Remark}

\makeatletter
\let\c@equation\c@thm
\makeatother
\numberwithin{equation}{section}

\def\Ind{\setbox0=\hbox{$x$}\kern\wd0\hbox to 0pt{\hss$\mid$\hss} \lower.9\ht0\hbox to 0pt{\hss$\smile$\hss}\kern\wd0} 

\def\Notind{\setbox0=\hbox{$x$}\kern\wd0\hbox to 0pt{\mathchardef \nn=12854\hss$\nn$\kern1.4\wd0\hss}\hbox to 0pt{\hss$\mid$\hss}\lower.9\ht0 \hbox to 0pt{\hss$\smile$\hss}\kern\wd0}

\title{A note on NSOP$_{1}$ in one variable}

\author{Nicholas Ramsey}

\date{\today}

\begin{document}

\begin{abstract}
We prove that, in order to establish that a theory is NSOP$_{1}$, it suffices to show that no formula in a single free variable has SOP$_{1}$.
\end{abstract}

\maketitle

\section{Introduction}

This note is concerned with showing strong order property 1 (SOP$_1$) is witnessed by a formula with only one free variable.  SOP$_{1}$ was introduced by Mirna D\v{z}amonja and Saharon Shelah in their study of the $\unlhd^{*}$-order and they observed the class of NSOP$_{1}$ contains the simple theories \cite{dvzamonja2004maximality}.  Our subsequent work with Artem Chernikov characterized NSOP$_{1}$ in terms of independent amalgamation of types, which gave a Kim-Pillay-style criterion for NSOP$_{1}$ that, in turn, implied that many non-simple examples of interest lie within this class \cite{ArtemNick}.  Later with Itay Kaplan, we introduced the theory of Kim-independence which provided evidence that NSOP$_{1}$ is a meaningful dividing line, admitting a structure theory close to simplicity theory \cite{kaplan2017kim}.  

SOP$_{1}$ is distinctive among dividing lines because of the difficulty of showing that a theory is NSOP$_{1}$ directly by syntactic means. In essentially all known examples of non-simple NSOP$_{1}$ theories, one first shows that the theory has a well-behaved notion of independence and then makes use of the Kim-Pillay-style criterion from \cite[Corollary 4.1]{ArtemNick} to show that this implies the theory is NSOP$_{1}$. In algebraic examples, such as Frobenius fields or bilinear forms over an algebraically closed field, this strategy is natural and closely parallels the established strategy for showing the simplicity of similar theories, such as bounded PAC fields or ACFA. However, in combinatorial examples, this approach can seem rather cumbersome or indirect.  

We simplify the syntax of SOP$_{1}$ by proving that SOP$_{1}$ is always witnessed by a formula in a single free variable.  One-variable theorems have been proved for almost all of the major dividing lines, both because it makes it easier to check whether a theory has the given property, and because it is a natural test question for one's understanding of the dividing line's behavior.   Yet these theorems can sometimes be difficult to discover.  For example, the questions of whether there are one-variable theorems for the strict order property or the tree property of the second kind were both posed as open problems by Shelah; the former was settled later by Lachlan \cite{lachlan1975remark}, the latter much later by Chernikov \cite{ChernikovNTP2}.  In some cases, the analysis can be simplified by considering generalized indiscernibles, e.g. indiscernible arrays for TP$_{2}$ or indiscernible trees for TP$_{1}$/SOP$_{2}$, but these are of little direct use in studying formulas witnessing SOP$_{1}$ (see \cite[p. 29n1]{Harr} for a discussion).  The argument below instead makes use of an equivalent formulation of SOP$_{1}$ in terms of a sequence of pairs to conclude by a direct combinatorial argument.   

%
%
%

\section{The proof}

We begin by noting some equivalent formulations of SOP$_{1}$ in terms of arrays that will be useful.  In referring to an array $(c_{i,j})_{i < \omega, j < 2}$, we write $\overline{c}_{i} = (c_{i,0},c_{i,1})$ and $\overline{c}_{<i} = (\overline{c}_{k})_{k < i}$.  We write $L(C)$ to denote the collection of $L$-formulas with parameters from the set $C$.  We always assume $T$ is a complete theory with monster model $\mathbb{M} \models T$.  

\begin{fact} \cite[Lemma 2.2, Proposition 2.4]{kaplan2017kim} \label{arrayequivalent}
The following are equivalent:
\begin{enumerate}
\item $T$ has SOP$_{1}$\textemdash that is, there is a formula $\varphi(x;y)$ and a tree of tuples $(a_{\eta})_{\eta \in 2^{<\omega}}$ so that 
\begin{enumerate}
\item For all $\eta \in 2^{\omega}$, $\{\varphi(x;a_{\eta | \alpha}) : \alpha < \omega\}$ is consistent.
\item For all $\nu \frown \langle 0 \rangle \unlhd \eta \in 2^{<\omega}$, $\{\varphi(x;a_{\eta}), \varphi(x;a_{\nu \frown \langle 1 \rangle}) \}$ is inconsistent.  
\end{enumerate}
\item There is a formula $\varphi(x;y)$, possibly with parameters from $C$, and array $(c_{i,j})_{i < \omega, j < 2}$ so that 
\begin{enumerate}
\item $c_{i,0} \equiv_{C\overline{c}_{<i}} c_{i,1}$ for all $i < \omega$.
\item $\varphi(x;c_{i,0}) : i < \omega\}$ is consistent.
\item $\{\varphi(x;c_{i,1}) :i < \omega\}$ is $2$-inconsistent.
\end{enumerate}
\item  There is an array \((c_{i,j})_{i < \omega, j < 2}\), with $c_{i,j} = (d_{ij},e_{ij})$ for all $i,j$, and formulas \(\chi_{1}(x;y)\) and $\chi_{2}(x;z)$, possibly with parameters from $C$, so that, writing $\psi(x;y,z)$ for $\chi_{1}(x;y) \wedge \chi_{2}(x;z)$, the following conditions are satisfied:  
\begin{enumerate}
\item For all \(i < \omega\), \(e_{i,0} \equiv_{C c_{<i,0}e_{<i,1}} e_{i,1}\);
\item \(\{\psi(x;c_{i,0}) : i < \omega\}\) is consistent; 
\item If \(j \leq i$ then $\{\chi_{1}(x;d_{i,0}), \chi_{2}(x;e_{j,1})\}\) is inconsistent.
\end{enumerate}
\end{enumerate}
\end{fact}

\begin{rem}
The equivalence of (1) and (2) is \cite[Proposition 2.4]{kaplan2017kim} and the implication (3)$\implies$(1) is proved in \cite[Lemma 2.2]{kaplan2017kim}.  To complete the equivalence, we explain the direction (2)$\implies$(3).  Given a formula $\varphi(x;y)$ over $C$ and an array $(c_{i,j})_{i<\omega, j < 2}$ as in (2), we define a new array $(c'_{i,j})_{i < \omega, j < 2}$ by $c'_{i,j} = (c_{i,j},c_{i,j})$ and formulas $\chi_{1}(x;y) = \varphi(x;y)$ and $\chi_{2}(x;z) = \varphi(x;z)$.  The conditions 3(a) and 3(b) are immediate from 2(a) and 2(b) respectively.  Moreover, if $j \leq i$, then $c_{i,0} \equiv_{C \overline{c}_{<i}} c_{i,1}$ so $\{\chi_{1}(x;c_{i,0}),\chi_{2}(x;c_{j,1})\}$ is inconsistent because $\{\varphi(x;c_{k,1}) : k < \omega\}$ is 2-inconsistent, which shows 3(c).    
\end{rem}

\begin{rem}
Although conditions (1)-(3) are not, in general, equivalent at the level of formulas, if one of the conditions is true for a formula $\varphi(x;y)$ with $l(x) = n$, then for any of the other conditions, there is a formula $\varphi'(x;y')$ witnessing this\textemdash that is, each condition will be satisfied with a formula of the same number of object variables, although the parameter variables may differ.  Hence we say $T$ has SOP$_{1}$ witnessed by a formula in a single free variable if there is a $\varphi(x;y)$ with $l(x) = 1$ for at least one of the conditions (1)-(3).  
\end{rem}

\begin{rem}
The proof of (1) from (2) or (3) gives a formula $\varphi(x;y)$ that has SOP$_{1}$, possibly with parameters from $C$.  Note, however, that this implies there is formula without parameters witnessing SOP$_{1}$:  if $c$ enumerates these parameters and $\varphi(x;y,c)$ witnesses SOP$_{1}$ via the tree of tuples $(a_{\eta})_{\eta \in 2^{<\omega}}$, then $\varphi(x;y,z)$ witnesses SOP$_{1}$ via the tree of tuples $(b_{\eta})_{\eta \in 2^{<\omega}}$ defined by $b_{\eta} = (a_{\eta},c)$.  Conversely, if $\varphi(x;y)$ witnesses SOP$_{1}$, it will also witness SOP$_{1}$ in the theory with $C$ named as constants, which allows one to deduce (2) and (3) over the set $C$ from the statement of \cite[Proposition 2.4]{kaplan2017kim}.  
\end{rem}

\begin{lem}\label{moreindiscernible}
If $T$ has SOP$_{1}$, there is a formula $\varphi(x;y)$ and an array $(c_{i,0},c_{i,1})_{i < \omega}$ so that 
\begin{enumerate}
\item $\{\varphi(x;c_{i,0}) : i < \omega\}$ is consistent.
\item $\{\varphi(x;c_{i,1}) :i < \omega\}$ is 2-inconsistent.
\item $(\overline{c}_{i})_{i < \omega}$ is an indiscernible sequence.
\item $c_{i,0} \equiv_{\overline{c}_{<i}} c_{i,1}$ for all $i < \omega$.
\item $(c_{k,0})_{k \geq i}$ is $\overline{c}_{<i}c_{i,1}$-indiscernible.
\end{enumerate}
\end{lem}

\begin{proof}
If $T$ has SOP$_{1}$, then by Fact \ref{arrayequivalent}(2), Ramsey, and compactness, there is a formula $\varphi(x;y)$ and an array $(b_{i,0},b_{i,1})_{i < \omega}$ satisfying (1)-(4) in the statement.  Define $c_{i,0} = b_{2i+1,0}$ and $c_{i,1} = b_{2i,1}$ for all $i < \omega$.  The array $(c_{i,0},c_{i,1})_{i < \omega}$ clearly satisfies (1)-(3).  (5) is also clear, since $(b_{k,0})_{k \geq 2i+1}$ is $\overline{b}_{<2i+1}$-indiscernible.  To see (4), note $b_{2i+1,0} \equiv_{\overline{b}_{<2i}} b_{2i+1,1}$ and hence $b_{2i+1,1} \equiv_{\overline{b}_{<2i}} b_{2i,1}$ by (3) so, by definition, $c_{i,0} \equiv_{\overline{c}_{<i}} c_{i,1}$.  
\end{proof}

\begin{lem}\label{mainstep}
Suppose $T$ does not witness SOP$_{1}$ with any formula in the variables $x$.  Suppose $b$ is a tuple of the same length as $x$, $C$ is some set of parameters, and $(c_{i,0},c_{i,1})_{i < \omega}$ is an array satisfying
\begin{enumerate}
\item $(\overline{c}_{i})_{i < \omega}$ is a $C$-indiscernible sequence.
\item $(c_{i,0})_{i < \omega}$ is $Cb$-indiscernible.
\item $c_{i,0} \equiv_{C\overline{c}_{<i}} c_{i,1}$ for all $i < \omega$.
\item $(c_{k,0})_{k \geq i}$ is $C\overline{c}_{<i}c_{i,1}$-indiscernible.
\end{enumerate}
Then there is $b' \equiv_{C(c_{i,0})_{i < \omega}} b$ such that $c_{i,0} \equiv_{Cb'} c_{i,1}$ for all $i < \omega$.  
\end{lem}

\begin{proof}
Suppose not.  Let $N$ be maximal so that 
$$
\text{tp}(b/C(c_{i,0})_{i < \omega}) \cup \{\varphi(x;c_{i,0}) \leftrightarrow \varphi(x;c_{i,1}) : i < N, \varphi \in L(C)\}
$$
is consistent.  By compactness, we may fix $\chi(x;c_{ \leq M,0}) \in \text{tp}(b/C(c_{i,0})_{i < \omega})$, a finite $\Delta(x) \subseteq \{\varphi(x;c_{i,0}) \leftrightarrow \varphi(x;c_{i,1}) : i < N, \varphi \in L(C)\}$, and a formula $\varphi \in L(C)$ so that 
$$
\chi(x;c_{\leq M,0}) \wedge \bigwedge \Delta(x) \vdash \varphi(x;c_{N,0}) \leftrightarrow \neg \varphi(x;c_{N,1}).
$$
We may take $M \geq N$, and, without loss of generality, $\chi(x;c_{\leq M,0}) \vdash \varphi(x;c_{N,0})$.  Put $C' = C \cup \overline{c}_{<N}$ and let 
\begin{eqnarray*}
a_{i.0} &=& (c_{(M-N)i + N,0}, c_{(M-N)i + N + 1,0}, \ldots, c_{(M-N)i+M,0}) \\
b_{i,0} &=& c_{(M-N)i+N,0} \\
b_{i,1} &=& c_{(M-N)i+N,1}.
\end{eqnarray*}
Unravelling definitions, we have $b_{i,0} \equiv_{C'a_{<i,0}\overline{b}_{<i}} b_{i,1}$ for all $i$.  Therefore, we may choose, for all $i < \omega$, some $a_{i,1}$ so that $a_{i,0}b_{i,0} \equiv_{C' a_{<i,0} \overline{b}_{<0}} a_{i,1}b_{i,1}$.  Set $z = (z_{0},\ldots, z_{M-N})$ and define $\psi(x;z) \in L(C')$ by
$$
\psi(x;z) = \chi(x;c_{< N,0}, z_{0},\ldots, z_{M-N}) \wedge \bigwedge \Delta(x).
$$

Since $\chi(x;c_{<N,0},c_{N,0},\ldots, c_{M,0}) \in \text{tp}(b/C(c_{i,0})_{i < \omega})$ and $(c_{i,0})_{i < \omega}$ is $Cb$-indiscernible, we have 
$$
\{\chi(x;c_{<N,0},c_{(M-N)i +N,0}, \ldots, c_{(M-N)i + M,0}) : i < \omega\} \subseteq \text{tp}(b/C(c_{i,0})_{i < \omega}).  
$$
By construction, $\Delta(x) \cup \text{tp}(b/C(c_{i,0})_{i < \omega})$ is consistent.  Unravelling definitions, we have $\{\psi(x;a_{i,0})\wedge \varphi(x;a_{i,0}) : i < \omega\}$ is consistent.  By our choice of $N$, we know 
$$
\chi(x;c_{<N,0},c_{N,0},\ldots, c_{M,0}) \wedge \bigwedge \Delta(x) \wedge \varphi(x;c_{N,0}) \wedge \varphi(x;c_{N,1})
$$
is inconsistent.  By $C$-indiscernibility of the sequence $(\overline{c}_{i})_{i < \omega}$, for all $i < \omega$, 
$$
\chi(x;c_{<N,0},c_{(M-N)i+N,0},\ldots, c_{(N-M)i+M,0}) \wedge \bigwedge \Delta(x) \wedge \varphi(x;c_{(M-N)i+N,0}) \wedge \varphi(x;c_{(M-N)i+N,1})
$$
is inconsistent.  Then as the sequence $(c_{k,0})_{k \geq (M-N)i+N}$ is $C\overline{c}_{<(M-N)i+N}c_{(M-N)i+N,1}$-indiscernible, it follows that, for all $j \geq i$, 
$$
\chi(x;c_{<N,0},c_{(M-N)j+N,0},\ldots, c_{(N-M)j+M,0}) \wedge \bigwedge \Delta(x) \wedge \varphi(x;c_{(M-N)j+N,0}) \wedge \varphi(x;c_{(M-N)i+N,1})
$$
is inconsistent.  Unravelling definitions again, this shows that for $i \leq j$, 
$$
\psi(x;a_{j,0}) \wedge \varphi(x;b_{j,0}) \wedge \varphi(x;b_{i,1})
$$ 
is inconsistent.  The formulas $\psi(x;z) \wedge \varphi(x;y)$ and $\varphi(x;y)$ witness the condition for SOP$_{1}$ in Fact \ref{arrayequivalent}(3) with respect to the array $((a_{i,j},b_{i,j}),b_{i,j})_{i < \omega, j < 2}$.  This shows $T$ has SOP$_{1}$ (in some formula in the variables $x$).  
\end{proof}

\begin{thm}
If $T$ has SOP$_{1}$, there is some formula in a single free-variable with SOP$_{1}$.    
\end{thm}

\begin{proof}
Suppose $\varphi(x,y;z)$ is a formula witnessing SOP$_{1}$.  So there is an array $(c_{i,0},c_{i,1})_{i < \omega}$ satisfying conditions (1)-(5) of Lemma \ref{moreindiscernible} with respect to $\varphi(x,y;z)$.  We will suppose $T$ does not witness SOP$_{1}$ in the free variables $y$ and we will exhibit a formula witnessing SOP$_{1}$ in the free variables $x$.  Let $C_{n}$ be the set enumerated by $\overline{c}_{<n}$.  For $n \leq m < \omega$, define the partial type $q_{n,m}(y)$ by 
$$
q_{n,m}(y) = \{\psi(y;c_{m,0}) \leftrightarrow \psi(y;c_{m,1}): \psi(y;z) \in L(C_{n})\}.
$$
By induction on $n < \omega$, we will choose $b_{n}$ so that 
\begin{enumerate}
\item $\{\varphi(x,b_{n};c_{i,0}) : i < \omega\}$ is consistent.
\item $b_{n}$ realizes $q_{k,m}$ for all $k < n$ and $m \geq k$.
\item $(\overline{c}_{k})_{k \geq n}$ is $C_{n}b_{n}$-indiscernible.  
\end{enumerate}
To begin, choose an arbitrary $(a_{0},b_{0}) \models \{\varphi(x,y;c_{i,0}) : i < \omega\}$.  By Ramsey, compactness, and automorphism, we may assume $(\overline{c}_{i})_{i < \omega}$ is $b_{0}$-indiscernible.  Next, suppose we are given $b_{n}$ so that $\{\varphi(x,b_{n};c_{i,0}) : i < \omega\}$ is consistent, $b_{n}$ realizes $q_{k,m}$ for all $k < n$ and $m \geq k$, and $(\overline{c}_{k})_{k \geq n}$ is $b_{n}C_{n}$-indiscernible.  We additionally know $c_{k,0} \equiv_{C_{n} \overline{c}_{<k}} c_{k,1}$ and $(c_{l,0})_{l \geq k}$ is $C_{n} \overline{c}_{<k} c_{k,1}$-indiscernible for all $k \geq n$.  Therefore, we may apply Lemma \ref{mainstep} to conclude 
$$
\text{tp}(b_{n}/C_{n}(c_{i,0})_{i < \omega}) \cup \{\psi(y;c_{m,0}) \leftrightarrow \psi(y;c_{m,1}) : m \geq n, \psi \in L(C_{n})\}
$$
is consistent.  Let $b_{n+1}$ realize this partial type.  By Ramsey, compactness, and an automorphism over $C_{n}$, we may assume $(\overline{c}_{k})_{k \geq n}$ is $C_{n}b_{n+1}$-indiscernible, hence $(\overline{c}_{k})_{k \geq n+1}$ is $C_{n+1}b_{n+1}$-indiscernible.  It follows, then, that $\{\varphi(x,b_{n+1};c_{i,0}) : i < \omega\}$ is consistent and $b_{n+1}$ realizes $q_{k,m}(y)$ for all $k < n+1$ and $m \geq k$.  

By compactness, then, we obtain a tuple $b$ so that $\{\varphi(x,b;c_{i,0}) : i < \omega\}$ is consistent and $b$ realizes $q_{n,k}(y)$ for all $n \leq k < \omega$.  It follows that $c_{n,0} \equiv_{\overline{c}_{<n}b} c_{n,1}$ for all $n$.  Setting $d_{i,j} = (b,c_{i,j})$ for all $i < \omega, j < 2$, we obtain an array $(d_{i,0},d_{i,1})_{i < \omega}$ so that $\{\varphi(x;d_{i,0}) : i < \omega\}$ is consistent, $\{\varphi(x;d_{i,1}) : i < \omega\}$ is inconsistent, and $d_{i,0} \equiv_{\overline{d}_{<i}} d_{i,1}$ for all $i$.  This shows there is a formula in the variables $x$ witnessing SOP$_{1}$.  This shows we may reduce the number of variables in the formula witnessing SOP$_{1}$ and, by induction, we conclude.  
\end{proof}

\subsection*{Acknowledgements}

This work constitutes part of our dissertation under the supervision of Thomas Scanlon, whom we would like to thank.  

\bibliographystyle{alpha}
\bibliography{ms.bib}{}

\end{document}